\documentclass[12pt]{amsart}
\usepackage{amsmath, amssymb, mathabx, enumitem, amsfonts, mathrsfs}
\usepackage{color}
\usepackage[colorlinks=true,linkcolor=blue,urlcolor=blue,citecolor=blue, pagebackref]{hyperref}
\usepackage[alphabetic]{amsrefs}
\usepackage{ulem}
\usepackage{tikz}

\usepackage{dynkin-diagrams}

\NewDocumentCommand\dyn{mom}
{\(\IfStrEq{#3}{}{#1_{n}}{#1_{#3}}^{\IfValueT{#2}{#2}}\) 
& \IfValueTF{#2}{\dynkin{#1}[#2]{#3}}{\dynkin{#1}{#3}} \\ }

\usepackage[margin=0.7in]{geometry}

\newcommand{\Q}{\mathbb{Q}}

\newcommand{\R}{\mathbb{R}}

\newcommand{\op}{\operatorname}

\newtheorem{theorem}{Theorem}[section]

\newtheorem{remark}[theorem]{Remark}

\newtheorem{corollary}[theorem]{Corollary}

\newtheorem{proposition}[theorem]{Proposition}

\newtheorem{lemma}[theorem]{Lemma}

\newtheorem{definition}[theorem]{Definition}
\newtheorem{definition/lemma}[theorem]{Definition/Lemma}

\newtheorem{notation}{Notation}

\title[Rays of Kostka cones]{Vertices of Intersection Polytopes and Rays of Generalized Kostka Cones}
\author{Marc Besson}
\address{Department of Mathematics, University of North Carolina, Chapel Hill, NC 27599}
\email{marmarc@live.unc.edu}
\author{Sam Jeralds}
\address{Department of Mathematics, University of North Carolina, Chapel Hill, NC 27599}
\email{sjj280@live.unc.edu}
\author{Joshua Kiers}
\address{Department of Mathematics, Ohio State University, 281 W Lane Ave, Columbus, OH 43201}
\email{kiers.2@osu.edu}

\begin{document}
\maketitle

\begin{abstract} 
Let $\mathscr{K}(G)$ be the rational cone generated by pairs $(\lambda, \mu)$ where $\lambda$ and $\mu$ are dominant integral weights and $\mu$ is a nontrivial weight space in the representation $V_{\lambda}$ of $G$. We produce all extremal rays of $\mathscr{K}(G)$ by considering the vertices of corresponding intersection polytopes $IP_{\lambda}$, the set of points in $\mathscr{K}(G)$ with first coordinate $\lambda$. We show that vertices of $IP_{\varpi_i}$ arise as lifts of vertices coming from cones $\mathscr{K}(L)$ associated to simple Levi subgroups possessing the simple root $\alpha_i$. As corollaries we obtain a complete description of all extremal rays, as well as polynomial formulas describing the numbers of extremal rays depending on type and rank.
\end{abstract}

\section{Introduction}
 Let $G$ be a simple, simply-connected linear algebraic group over $\mathbb{C}$. We choose a maximal torus and Borel subgroup $T \subset B \subset G$. Then the irreducible finite-dimensional representations of $G$ are indexed by the dominant integral weights of $T$. For such a weight $\lambda$, the associated representation, $V_\lambda$, possesses a weight space decomposition with respect to $T$: $V_\lambda = \oplus V_\lambda(\mu)$, where $V_\lambda(\mu)$ is the subspace upon which $T$ acts through scalar multiplication by the character $\mu$. 
 
As is well-known, $V_\lambda(\mu)\ne (0)$ if and only if 
\begin{enumerate}[label=(\alph*)]
\item $\lambda-\mu$ lies in the root lattice for $G$ and
\item $\mu$ is contained inside the {\it Weyl polytope} $\op{conv}(W\cdot \lambda)$,
\end{enumerate}
where $W$ denotes the Weyl group and $W\cdot \lambda$ its orbit through $\lambda$. However, there is a simpler criterion if $\mu$ is already known to be a dominant weight (cf. \cite{St}): for $\lambda,\mu$ both dominant, $V_\lambda(\mu)\ne (0)$ if and only if 
\begin{enumerate}
\item [(a$'$)]$\lambda - \mu$ is a linear combination of simple roots with nonnegative integral coefficients. 
\end{enumerate}
Since $\dim V_\lambda(\mu) = \dim V_\lambda(w\mu)$ for any $w\in W$, restricting our attention to $\mu$ dominant does not, in fact, lose any information about the representation $V_\lambda$. It is customary to write $\lambda \succeq \mu$ for the statement $V_\lambda(\mu)\ne (0)$ and to call $\succeq$ the {\it dominance order.}

\subsection{Rays of the Kostka cone}

In type $A$, the multiplicity $m_{\lambda,\mu} = \dim V_\lambda(\mu)$ has classically been called a {\it Kostka number} and has meaning in a variety of contexts, such as symmetric functions and representations of the symmetric group, cf. \cite{Ful}.
We set $\mathscr{K}(G) = \{(\lambda,\mu) | \lambda,\mu \text{ dominant, } V_\lambda(\mu)\ne (0)\}$ and call $\mathscr{K}(G)$ the {\it Kostka cone} for $G$. In \cite{GKOY}, S. Gao, G. Orelowitz, A. Yong, and the third author completely describe the extremal rays of $\mathscr{K}(GL_n)$ using the language of partitions and Young tableaux and make further investigations into the Hilbert basis of the Kostka cone. Inspired by that work, we now seek to extend one of their results:
$$
\text{ {\bf Question:} What are the extremal rays of $\mathscr{K}(G)$, for $G$ simple and simply-connected?}
$$

Let us ensure that the above question makes sense.
If $X^*(T)$ denotes the space of all weights of $T$, criterion (a$'$) tells us that $\mathscr{K}(G)\subset (X^*(T))^2$ is exactly the solution space to a system of linear inequalities which we list in Proposition \ref{steinberg}.
In the ambient vector space $(X^*(T)\otimes \Q)^2$, the set $\mathscr{K}(G)_\Q$ of {\it rational} solutions to 
those inequalities thus forms a pointed, rational, polyhedral cone, and it makes sense to ask what its extremal rays are.

Consider momentarily the related cone $\mathscr{K}'(G) = \{(\lambda,\mu) | \lambda \text{ dominant, } V_\lambda(\mu)\ne (0)\}$. The two cones are related by $\mathscr{K}'(G) = W \mathscr{K}(G)$, where $W$ acts on the second factor only: $w.(\lambda,\mu) = (\lambda,w\mu)$. Given a dominant weight $\lambda$, the affine slice $\{(\lambda,\cdot)\}\cap \mathscr{K}'(G)_\Q$ is simply $\op{conv}(W\cdot \lambda)$. Since the vertices of $\op{conv}(W\cdot \lambda)$ depend linearly on $\lambda$, a straightforward argument (cf. Propositions \ref{somerays} and \ref{firstcoord}) shows that the extremal rays of $\mathscr{K}'(G)$ consist of the various $(\varpi_i,w\varpi_i)$ as $\varpi_i$ ranges over the set of fundamental weights and $w$ over $W$. 
 
The extremal rays of $\mathscr{K}(G)$ cannot simply be determined from those of $\mathscr{K}'(G)$, even though the former is naturally a quotient of the latter by $W$ (the problem is that this quotient is not operation-preserving at the level of semigroups.) 
Nonetheless, the technique for $\mathscr{K}'(G)$ can be used for $\mathscr{K}(G)$. 
Given $\lambda$, the affine slice $\{(\lambda,\cdot)\}\cap\mathscr{K}(G)_\Q$ is some convex polytope $IP_\lambda$; it can be viewed as the intersection of $\op{conv}(W\cdot\lambda)$ with the dominant chamber of $X^*(T)_\Q$. Examining these polytopes, we find an averaging formula for the vertices of $IP_\lambda$ (Proposition \ref{avg}), thus establishing that they depend linearly on $\lambda$, which allows us to conclude:

\begin{theorem}
The extremal rays of $\mathscr{K}(G)_\Q$ are all of the form $(\lambda,\mu)$ where $\lambda = \varpi_j$ is a fundamental weight and $\mu$ is a vertex of $IP_{\varpi_j}$. Conversely, every such pair $(\lambda,\mu)$ produces an extremal ray of $\mathscr{K}(G)_{\mathbb{Q}}$. 
\end{theorem}

Thus in order to completely describe the set of extremal rays of $\mathscr{K}(G)_{\mathbb{Q}}$, it is sufficient for us to enumerate the vertices of all of the intersection polytopes $IP_{\varpi_j}$ for all fundamental weights $\varpi_j$ of $T$. We complete this work in Sections \ref{desc} and \ref{enum}, with the following result. 

\begin{theorem}\label{maine}
The vertices of $IP_{\varpi_j}$ consist of the following: 
\begin{enumerate}[label=(\alph*)]
\item $\varpi_j$, and 
\item $\varpi_j - \sum_I c_i\alpha_i$, where $I$ stands for a connected subdiagram of the Dynkin diagram for $G$ containing node $j$ and the coefficients $c_i$ are the entries of the $j^\text{th}$ column of the inverse transpose of the Cartan matrix associated to $I$.
\end{enumerate}
\end{theorem}

\begin{remark}
It is natural to consider (a) as a special case of (b) where $I$ is the empty subdiagram. 
\end{remark}

\subsection{Two examples}

First take $G=SL_2$. For nonnegative integers $\ell,m$, the weight space $V_\ell(m)\ne (0)$ if and only if $\ell\ge m$ and $\ell-m$ is even. So the cone $\mathscr{K}(SL_2)$ looks like this:
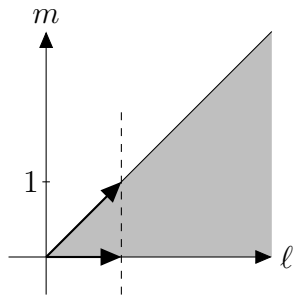
\begin{center}
\begin{tikzpicture}[>=triangle 45]
\fill [lightgray] (0,0) -- (3,3) -- (3,0) -- (0,0);
\draw [->] (-0.5,0) -- (3,0);
\draw [->] (0,-0.5) -- (0,3);
\draw [thick, ->] (0,0) -- (1,0);
\draw [thick,->] (0,0) -- (1,1);
\draw [dashed ] (1,-0.5) -- (1,2);
\draw [-] (0,0) -- (3,3);
\draw [-] (-0.05,1) -- (0.05,1);
\node at (3.2, 0) {$\ell$};
\node at (0, 3.2) {$m$};
\node at (-0.2,1) {$1$};
\end{tikzpicture}
\end{center}
Notice that, intersecting the cone with the dashed line through $\ell = 1$, we obtain the polytope isomorphic to the interval $[0,1]$, whose two vertices give extremal rays as depicted. 


Now take $G = Sp_2$ (type $C_2$). The intersection polytope $IP_\lambda$ for $\lambda = \varpi_1+\varpi_2$ has $4=2^2$ vertices as predicted by Remark \ref{regularremark}.

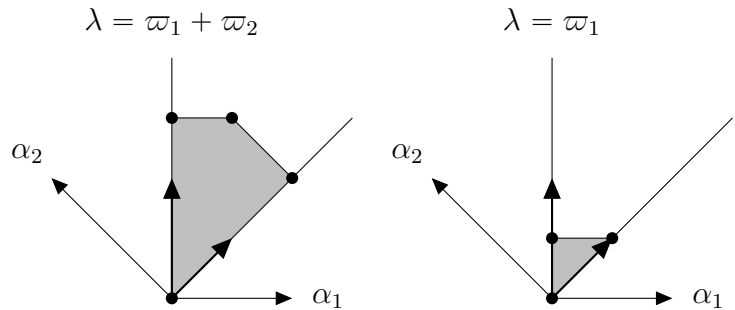
\begin{center}
\begin{tabular}{cc}
\begin{tikzpicture}[scale=1.6,>=triangle 45]
\fill [lightgray] (0,0) -- (0,1.5) -- (0.5,1.5) -- (1,1) -- (0,0);
\draw[->] (0,0) -- (1,0);
\draw[->] (0,0) -- (-1,1);
\draw[thick, ->] (0,0) -- (0.5,0.5);
\draw[thick, ->] (0,0) -- (0,1);
\draw[-] (0,0) -- (1.5,1.5);
\draw[-] (0,0) -- (0,2); 
\draw[-] (0,1.5) -- (0.5,1.5);
\draw[-] (0.5,1.5) -- (1,1);
\node at (1.3,0) {$\alpha_1$};
\node at (-1.2,1.2) {$\alpha_2$};
\fill (0,0) circle [radius = 0.05];
\fill (0,1.5) circle [radius = 0.05];
\fill (0.5,1.5) circle [radius = 0.05];
\fill (1,1) circle [radius = 0.05];
\node at (0,2.3) {$\lambda = \varpi_1+\varpi_2$};
\end{tikzpicture} & 
\begin{tikzpicture}[scale=1.6,>=triangle 45]
\fill [lightgray] (0,0) -- (0,0.5) -- (0.5,0.5) -- (0,0);
\draw[->] (0,0) -- (1,0);
\draw[->] (0,0) -- (-1,1);
\draw[thick, ->] (0,0) -- (0.5,0.5);
\draw[thick, ->] (0,0) -- (0,1);
\draw[-] (0,0) -- (1.5,1.5);
\draw[-] (0,0) -- (0,2); 
\draw[-] (0,0.5) -- (0.5,0.5);
\node at (1.3,0) {$\alpha_1$};
\node at (-1.2,1.2) {$\alpha_2$};
\fill (0,0) circle [radius = 0.05];
\fill (0,0.5) circle [radius = 0.05];
\fill (0.5,0.5) circle [radius = 0.05];
\node at (0,2.3) {$\lambda = \varpi_1$};
\end{tikzpicture}
\end{tabular}
\end{center}

In contrast, the polytope $IP_{\varpi_1}$ has the three vertices $\varpi_1$, $\frac{1}{2}\varpi_2 = \varpi_1 - \frac{1}{2}\alpha_1$, and $0 = \varpi_1 - \frac{1}{2}(2\alpha_1 +\alpha_2)$; note that $\frac{1}{2}(1)$ and $\frac{1}{2}\left(\begin{array}{c}2 \\ 1\end{array}\right)$ are indeed the corresponding columns of the inverse transposes of the Levi Cartan matrices. 



\subsection{Levi induction}

To each vertex $v$ of an intersection polytope $IP_\lambda$, we assign a Levi subgroup $L\subseteq G$ and show that $v$ can be lifted from a corresponding vertex in an intersection polytope for $L$. This lifting procedure induces a map $\op{Ind}:\mathscr{K}(L)\to \mathscr{K}(G)$ that takes extremal rays to extremal rays. We describe this procedure in Section \ref{levilifting}. It is tempting to compare this to the induction of extremal rays for the eigencone from Levi subgroups described in \cite{Belkale,BKiers}. 


\subsection{Formulas for the numbers of extremal rays}

We also count the number of extremal rays of $\mathscr{K}(G)$ and produce, for each Lie type, a polynomial formula as a function of the rank. As a consequence of Theorem \ref{maine}, these formulas do {\it not} depend on the lacing of the associated Dynkin diagrams, so there are only three cases to consider: types $A_r$, $D_r$, and $E_r$. We also point out that these counting polynomials each begin with leading term $r^3/6$. See Remark \ref{ADE}. 

\subsection{Acknowledgements} This paper owes its existence to A. Yong, who posed this question for $GL_n$ and studied it in \cite{GKOY}. We thank A. Yong and P. Belkale for providing useful feedback on an earlier version of this manuscript.

\section{Notation and Background}

We fix $G$ a simple, simply-connected linear algebraic group over $\mathbb{C}$. We choose a maximal torus and Borel subgroup $T \subset B \subset G$. We denote by $X^*(T)$ the lattice of weights of $T$, and by $X_*(T)$ the lattice of coweights. Their natural pairing is denoted by $\langle \, , \rangle$. We let $\Phi$ denote the set of  roots of $G$ with respect to $T$, and denote by $\Phi^+$ the set of positive roots of $G$ with respect to $B$.  We let $\Phi^{\vee}$ denote the set of coroots, so $(\Phi, X^*(T), \Phi^{\vee}, X_*(T))$ is a root datum for $G$. For a subset $I$ of $\{1,\hdots,r\}$, denote by $L_I$ the semisimple part of the corresponding Levi subgroup, where $\alpha_i \in \Phi(L_I)$ for all $i \in I$. We write $\Lambda^+$ for the set of dominant weights. We denote by $W$ the Weyl group of $G$. We denote by $\Delta=\{\alpha_1, \dots \alpha_r\}$ (resp., $\{\alpha^{\vee}_1, \dots \alpha^{\vee}_r\} $) the set of simple roots in  $\Phi$ (resp., simple coroots in $\Phi^{\vee}$). The fundamental coweights $x_i\in X_*(T)\otimes \Q$ satisfy $\langle \alpha_i, x_j \rangle = \delta_{ij}$. 
 
 If $\lambda \in \Lambda^+$, we write $V_{\lambda}$ for the irreducible representation of $G$ of highest weight $\lambda$. If $\mu \in X^*(T)$ we write $V_{\lambda}(\mu)$ for the subspace of weight $\mu$.
 Given a dominant weight $\lambda \in \mathfrak{h}_{\mathbb{Q}}^*$, we can associate to it the Weyl polytope $\mathrm{conv}(W \cdot \lambda)$, which we denote by $W_{\lambda}$.
 For an $n-$dimensional polytope $P$, we call a face of dimension $n-1$ a facet.
 
\begin{definition}
	Denote by $\mathscr{K}(G)$ the set of pairs $(\lambda, \mu) \in \Lambda^+ \times \Lambda^+$ such that $V_{\lambda}(\mu) \neq 0$. Let $\mathscr{K}(G)_{\mathbb{Q}^+}= \mathscr{K}(G) \otimes_{\mathbb{Z}^{\geq 0}} \mathbb{Q}^{\geq 0}$. Since we work over $\mathbb{Q}$ for this whole paper, we will abuse notation and write $\mathscr{K}(G)$ for $\mathscr{K}(G)_{\mathbb{Q}^+}$.
\end{definition}

The following proposition is well-known (see \cite{St}):

\begin{proposition}\label{steinberg}
	$\mathscr{K}(G)$ is a rational cone. Moreover, $(\lambda,\mu)\in \mathscr{K}(G)$ if and only if the inequalities 
	\[
	\langle \lambda-\mu,x_i\rangle \ge 0
	\]
	\[\langle \lambda, \alpha^{\vee}_i \rangle \geq 0 \]
	\[\langle \mu, \alpha^{\vee}_i \rangle \geq 0 \]
for each $i\in \{1 ,\dots r\}$ are satisfied. 
\end{proposition}

%

\section{Vertices of the Intersection Polytope}
In order to study the extremal rays of the Kostka cone associated to $G$, we study the following associated polytope.

\begin{definition}
	We define the intersection polytope  $IP_{\lambda}:= \Lambda^+_\Q \cap W_{\lambda}$.
\end{definition}

This is clearly a convex polytope, as it is the intersection of the convex polytope $W_{\lambda}$ with a finite collection of half-spaces. The facets of $IP_{\lambda}$ fall into two classes. 

The first class of facets are sub-polytopes of the hyperplanes defining $\Lambda^+_\Q$.  These facets $F^{\lambda}_i$, associated to simple coroots, are defined as the intersection 

\begin{equation}F^{\lambda}_i= IP_{\lambda} \cap \{\mu \in \mathfrak{h}^*_{\mathbb{Q}} | \langle \mu, \alpha^{\vee}_i \rangle =0\}.
\end{equation}  

Similarly, we have a second class of facets $E^{\lambda}_j$, associated to fundamental coweights, which are sub-polytopes of some of the facets of $W_{\lambda}$, defined by
\begin{equation}
E^{\lambda}_j=IP_{\lambda} \cap \{\mu \in \mathfrak{h}^*_{\mathbb{Q}} | \langle \lambda-\mu, x_j \rangle =0 \}.
\end{equation}

The candidates for vertices of $IP_{\lambda}$ are suitable intersections $(\cap F^{\lambda}_i) \cap (\cap E^{\lambda}_j)$.

First we explore why it is at all reasonable to expect these intersections to be well-behaved.

\begin{proposition}\label{basis}
Suppose $I \sqcup J =\{1,\hdots,r\}$ (disjoint union).  Then the collection 
$$
\{\alpha_i^\vee\}_{i\in I}\bigcup \{x_j\}_{j\in J}
$$
form a basis of $\mathfrak{h}$. 
\end{proposition}

\begin{proof}
If we can show they are linearly independent, that will be sufficient. Assume there is a relation 
$$
\sum a_i\alpha_i^\vee+\sum b_jx_j = 0.
$$
Then for every $i_0\in I$, 
$$
\alpha_{i_0}\left( \sum a_i\alpha_i^\vee\right) = 0.
$$
Let $L_I$ be the Levi associated to $I$. 
Then the element $\sum a_i\alpha_i^\vee\in \mathfrak{h}_{L_I}$ must be identically $0$ (since the $\alpha_i,i\in I$ form a basis of $\mathfrak{h}_{L_I}^*$). So each $a_i = 0$ in our relation. But 
$$
\sum b_j x_j = 0
$$
forces every $b_j = 0$ since the $x_j$ are linearly independent. 
\end{proof}

We use the following lemma from \cite{LusTits}:
\begin{lemma}\label{prakash}
Let $G$ be be semisimple. Let $\lambda$ be a dominant weight. Then for any $j$, $\langle \lambda, x_j \rangle \geq 0$. Furthermore, if $G$ is simple, this is strict.
\end{lemma}

\begin{definition}
	Suppose $I,J\subseteq \{1,\hdots,r\}$, and take $x \in \mathfrak{h}^*_{\mathbb{Q}}$ and $\lambda$ to be a dominant weight. If the system of equations 
	\begin{align*}
	\langle x,\alpha_i^\vee\rangle & = 0, ~~ i\in I\\
	\langle\lambda-x,x_j\rangle &= 0, ~~ j\in J
	\end{align*}
	has a unique solution, we will denote it by $v_{I,J}$. 
\end{definition}

\begin{lemma}\label{iPrime}
	The solutions $v_{I,J}$ where $I \sqcup J = \{1, \dots, r\}$ are vertices of $IP_\lambda$.
\end{lemma}
%
%

\begin{proof}
By Proposition \ref{basis}, there exists a unique solution in $\mathfrak{h}_\Q^*$ to the system of equations 
\begin{align*}
	\langle x,\alpha_i^\vee\rangle & = 0, ~~ i\in I\\
	\langle\lambda-x,x_j\rangle &= 0, ~~ j\in J.
\end{align*}
To ensure that $v_{I,J}$ is indeed inside $IP_\lambda$, it must satisfy the following two additional systems of inequalities: (a) $i\not\in I\implies \langle v_{I,J},\alpha_i^\vee\rangle\ge 0$ and (b) $j\not \in J\implies \langle \lambda - v_{I,J},x_j\rangle \ge 0$. The $J$ equations tell us that 
$$
\lambda = v_{I,J} + \sum_{i\in I} a_i\alpha_i
$$
for suitable rational numbers $a_i$. To establish (b), we must show that each $a_i\ge 0$. Indeed, restricted as weights for $L=L_I$, we have agreement 
$$
\lambda\big|_{\mathfrak{h}_{L}} = \sum_{i\in I} a_i \alpha_i\big|_{\mathfrak{h}_{L}}. 
$$
Furthermore, the $\alpha_i \big|_{\mathfrak{h}_{L}}$ are the simple roots of the $L$ root system and $\lambda\big|_{\mathfrak{h}_{L}}$ is still dominant. Therefore, by Lemma \ref{prakash} applied to $L$, each $a_i\ge 0$. Given that (b) holds, we can now verify (a): if $i \not \in I$ then 
$$
\langle v_{I,J},\alpha_i^\vee\rangle  = \langle \lambda,\alpha_i^\vee\rangle +\sum_{i'\in I}a_{i'} \langle -\alpha_{i'},\alpha_i^\vee\rangle \ge 0 
$$
since $\langle \alpha_{i'},\alpha_i^\vee\rangle \le 0$ when $i'\ne i$. 
\end{proof}

Next we will prove that any vertex of $IP_{\lambda}$ is of the form $v_{I,J}$ for suitable $I,J$ satisfying $I \cup J= \{1, \dots r\}$ and $I \cap J = \varnothing$, establishing Theorem \ref{verticestheorem} stated below. Certainly any vertex of $IP_{\lambda}$ is the intersection of the facets on which it lies, so every vertex is the unique solution to a system of the following form:
\begin{align}\nonumber
x&\in IP_\lambda\\\label{system0}
x&\in F_i^\lambda~\forall i\in I\\\nonumber
x&\in E_j^\lambda~\forall j\in J
\end{align}
for some $I,J\subseteq \{1,\hdots,r\}$, not necessarily disjoint. Note that a vertex $v$ satisfying the above system may satisfy yet more equalities of the form $\langle \lambda- v, x_l \rangle=0$ for $l \notin J$. For any such further equality satisfied by $v$, we add $l$ to $J$ until $J$ is maximal, so that $J=\{j|x\in E_j^\lambda\}$.
We now show that we can always assume $I\sqcup J = \{1,\hdots, r\}$. We begin with the following lemma.

\begin{lemma}\label{disj}
Suppose $I,J\subseteq\{1,\hdots,r\}$ and that $v$ is the unique solution to (\ref{system0}), with $J$ maximal. 
If $k \in I\cap J$, then $\langle\lambda,\alpha_k^\vee\rangle = 0$. Furthermore, $v$ is the unique solution to the weaker system
\begin{align}\nonumber
x&\in IP_\lambda\\\label{system1}
x&\in F_i^\lambda~\forall i\in I-\{k\}\\\nonumber
x&\in E_j^\lambda~\forall j\in J.
\end{align}
\end{lemma}

\begin{proof}

Take $k\in I\cap J$. By $v\in IP_\lambda$ and $v\in E_j^\lambda$ for all $j\in J$, we may write
$$
v = \lambda-\sum_{\ell\not\in J} k_\ell \alpha_\ell
$$
for some rational numbers $k_\ell\ge 0$. 
Furthermore, since $J$ is assumed to be maximal, we know each $k_\ell>0$. Because $k\in J$, $\langle \alpha_\ell,\alpha_k^\vee\rangle \le 0$ for all $\ell\not\in J$. Therefore,  given that $\langle v,\alpha_k^\vee\rangle = 0$,
$$
0\le \langle \lambda,\alpha_k^\vee\rangle   = \langle v,\alpha_k^\vee\rangle + \sum_{\ell\not\in J}k_\ell\langle \alpha_\ell,\alpha_k^\vee\rangle \le 0,
$$
so all expressions appearing are $0$. In particular, $\langle \lambda,\alpha_k^\vee\rangle = 0$. Furthermore, each $\langle \alpha_\ell,\alpha_k^\vee\rangle = 0$ since each $k_\ell>0$.



Now take any $v'$ satisfying (\ref{system1}). Again we can write $v'=\lambda- \sum_{i \notin J} b_i \alpha_i$. Note immediately that by the above conditions on $\lambda$ and $\alpha_i$ for $i \notin J$, $\langle v', \alpha_k^{\vee} \rangle =0$. Thus $v'$ satisfies (\ref{system0}), and so $v'=v$.
\end{proof}

Now let $v$ be an arbitrary vertex of $IP_{\lambda}$. Then for some $I,J\subseteq \{1,\hdots,r\}$, $v$ is defined by the following properties:
\begin{align}\nonumber
v&\in \left(\bigcap_{j\in J}E_j^\lambda\right)\cap \left(\bigcap_{i\in I} F_i^\lambda\right)\\\label{props}
\langle v,\alpha_i^\vee\rangle &\ge 0, i \not\in I\\\nonumber
\langle \lambda-v,x_j\rangle &\ge 0, j\not \in J
\end{align}
By Lemma \ref{disj}, we assume $I\cap J = \emptyset$ and $J$ is maximal. If $I\cup J = \{1,\hdots,r\}$, then we are done. Otherwise, set $K = \{1,\hdots,r\} - (I\cup J)$. From Lemma \ref{iPrime}, the point $v_{I\cup K,J}$ is a vertex of $IP_\lambda$; furthermore $v_{I\cup K,J}$ satisfies the system (\ref{props}), so $v = v_{I\cup K,J}$. This completes the proof of the following

\begin{theorem}\label{verticestheorem}
	Let $\lambda\in \Lambda^+$. The vertices of $IP_{\lambda}$ are exactly the points $v_{I,J}$ for $I,J$ satisfying $I\sqcup J = \{1,\hdots, r\}$. 
\end{theorem}

Finally, we make use of the inequalities of Theorem \ref{verticestheorem}, which completely classify the vertices of $IP_\lambda$, to give a concise formula for a given $v_{I,J}(\lambda)$ in the following  

\begin{proposition}\label{avg}
Suppose $I,J$ are disjoint and $I\sqcup J =  \{1,\hdots,r\}$. Let $W_I$ be the Weyl subgroup of $W$ generated by the simple reflections $s_i, i\in I$. Then we have 
$$
v_{I,J}(\lambda) = \frac{1}{|W_I|}\sum_{w\in W_I} w\lambda.
$$
\end{proposition}

\begin{proof}

Set $v= \frac{1}{|W_I|}\sum_{w\in W_I} w\lambda$. First, for each $w \in W_I$, note that $w\lambda$ has an expansion $\lambda-\sum_{i \in I} k_{i} \alpha_{i}$ for some suitable nonnegative integers $k_{i}$; averaging across $W_I$, we therefore see that $\lambda-v$ is a nonnegative rational combination of simple roots $\alpha_i$ for $i \in I$. Second, for any $i \in I$, note that $s_i(v)=v$; therefore since the pairing is $W$-invariant, we have $\langle v, \alpha_i^\vee \rangle = \langle v, -\alpha_i^\vee \rangle =0$. So $v$ satisfies the system of equations defining $v_{I,J}$. 
\end{proof}


\begin{remark}\label{Averts}
In type $A$, the vertices of the polytopes $IP_\lambda$ were enumerated by Hoffman \cite{Hoff} while reproducing a theorem of Hardy, Littlewood, and P\'olya which characterizes dominance order using doubly-stochastic matrices.
\end{remark}

\begin{remark}\label{regularremark} For $G$ of any type, note that if $\lambda$ is regular dominant, the $v_{I, J}$ are all distinct for different pairs $(I,J)$ satisfying $I\sqcup J = \{1,\hdots,r\}$; in particular there are $2^r$ vertices. This follows from Lemma \ref{disj}. 

This is also recorded in \cite{BZ}*{Proposition 13 and following remarks}, attributed to Kostant, for the case $\lambda=2\rho$. Further, in loc. cit. the vertices of $IP_{2\rho}$ are given by $v_{I,J}=\rho+w_I\rho$, where $w_I \in W_I$ is the longest word. It is straightforward to check that this agrees with Proposition \ref{avg} for $\lambda=2\rho$. While the remarks in loc. cit. say that the number of vertices can be extended to any regular dominant weight $\lambda$ in place of $2\rho$, they do not give a similar formula for the vertices of $IP_\lambda$. Thus the content of Proposition \ref{avg} might be known to experts; however, we have not seen it in the literature. 
\end{remark}

As a consequence of Proposition \ref{avg}, we immediately obtain the following corollary, which allows us to relate vertices of distinct polytopes, and establishes that $IP_{\lambda + \mu}$ is the Minkowski sum of $IP_\lambda$ and $IP_\mu$. 

\begin{corollary}\label{linear}
	The vertices $v_{I,J}(\lambda)$ of $IP_{\lambda}$ depend linearly on $\lambda$.
\end{corollary}

\section{First description of the Extremal Rays}\label{desc}

We begin by naming {\it some} extremal rays of $\mathscr{K}(G)$. 

\begin{proposition}\label{somerays}
Suppose $I\sqcup J = \{1,\hdots, r\}$. Then $(\varpi_i,v_{I,J}(\varpi_i))$ gives an extremal ray of $\mathscr{K}(G)$. 
\end{proposition}

\begin{proof}
Suppose $(\varpi_i,v_{I,J}(\varpi_i)) = (\lambda_1,\mu_1)+(\lambda_2,\mu_2)$ where each $(\lambda_j,\mu_j)\in \mathscr{K}(G)$. We wish to show the $(\lambda_j,\mu_j)$ are parallel. Since the $\lambda_j$ are both dominant and sum to $\varpi_i$, we must have $\lambda_1 = a_1\varpi_i$ and $\lambda_2 = a_2\varpi_i$ where $a_1,a_2\ge 0$ and $a_1+a_2=1$. It follows that $\mu_1$ is inside the polytope $IP_{a_1\varpi_i}$ and $\mu_2$ inside $IP_{a_2\varpi_i}$. Since $a_1=0$ makes $(\lambda_1,\mu_1)=(0,0)$, in which case the pair $(\lambda_j,\mu_j)$ are trivially parallel, we may assume $a_1>0$. Likewise, we assume $a_2>0$. By scaling, $\frac{1}{a_j}\mu_j$ for $j=1,2$ belongs to the polytope $IP_{\varpi_i}$. Furthermore, the vertex $v_{I,J}(\varpi_i)$ is equal to the convex sum
$$
a_1 \left(\frac{1}{a_1} \mu_1\right) + a_2\left( \frac{1}{a_2} \mu_2\right)
$$
of elements in $IP_{\varpi_i}$. By the following lemma (a standard result in convex geometry whose proof we will omit), this forces $\frac{1}{a_j}\mu_j = v_{I,J}(\varpi_i)$ for each $j=1,2$. 
\end{proof}

\begin{lemma}
Let $P$ be a compact, convex polytope inside of $\R^n$ defined as the solution space to a system of linear inequalities (given by linear functions $f_j:\R^n\to \R,j=1,\hdots,m$):
$$
P = \{\vec x\in \R^n| f_j(\vec x)\ge 0~\forall j\}
$$
Let $v\in P$ be a vertex 
 If $v=t_1x_1+t_2x_2$ for points $x_i\in P$ and positive real numbers $t_i$ such that $t_1+t_2=1$, then each $x_i = v$. 
\end{lemma}



To complete the initial enumeration of all extremal rays of $\mathscr{K}(G)$, we now demonstrate that every extremal ray has the special form above. 





\begin{proposition}\label{firstcoord}
	If $(\lambda, \mu)$ is an extremal ray of $\mathscr{K}(G)$, then, up to scaling, $\lambda=\varpi_i$ for some $i$ and $\mu = v_{I,J}(\varpi_i)$ for some $I,J$ as above. 
\end{proposition}

\begin{proof}
	Let $(\lambda, \mu) \in \mathscr{K}(G)$. Since $\mu \in IP_{\lambda}$, we can write $\mu=\sum a_{I,J}v_{I,J}(\lambda)$ where $\sum a_{I,J}=1$. Moreover $\lambda=\sum a_{I,J} \lambda$. Thus we can rewrite $(\lambda, \mu)=(\sum a_{I,J}\lambda, \sum a_{I,J}v_{I,J}(\lambda))=\sum a_{I,J}(\lambda, v_{I,J}(\lambda)).$ Now we write $\lambda=\sum b_k \varpi_k$, and using that the $v_{I,J}(\lambda)$ are linear in $\lambda$ (Corollary \ref{linear}) we have 
	\begin{align*}
	(\lambda, \mu) &= \sum a_{I,J} (\lambda, v_{I,J}(\lambda)) \\
	&= \sum a_{I,J} (\sum b_k \varpi_k, v_{I,J}(\sum b_k \varpi_k)) \\
	&= \sum a_{I,J} (\sum b_k \varpi_k, \sum b_k v_{I,J}( \varpi_k)) \\
	&= \sum a_{I,J} \sum b_k (\varpi_k, v_{I,J}(\varpi_k)).
	\end{align*}
Since $(\lambda,\mu)$ is extremal, it is parallel to every $(\varpi_k, v_{I,J}(\varpi_k))$ such that $a_{I,J}b_k\ne 0$. Hence there is only one nonzero $b_k$ and the collection of pairs $I,J$ such that $a_{I,J}\ne 0$ all produce the same vertex $v_{I,J}(\varpi_k)$. 
\end{proof}


\section{Enumeration of the Extremal Rays}\label{enum}

Thus the generators of extremal rays of $\mathscr{K}(G)$ coincide with the collection of $(\varpi_i, v_{I,J}(\varpi_i)).$ There could be redundancy among the $v_{I,J}(\varpi_i)$s (indeed there is), so we now work toward enumerating these extremal rays without repetition. For the remainder of this section, fix an index $i$. 


We first take note of the redundancy of vertices $v_{I,J}(\varpi_i)$ for which $i\not\in I$. 
\begin{lemma}\label{mt}
If $i\not\in I$, then $v_{I,J}(\varpi_i) = \varpi_i$. 
\end{lemma}

\begin{proof}
This follows immediately from the observation that $W_I$ fixes $\varpi_i$, and Proposition \ref{avg}. 
\end{proof}

So from now on, assume that $i\in I$. If the Levi subdiagram corresponding to $I$ is not simple, we will show that the vertex $v_{I,J}(\varpi_i)$ is the same as that determined by the simple subLevi containing node $i$. 

\begin{lemma}\label{simpleton}
	Suppose that the Levi subdiagram corresponding to $I$ breaks up into simple (i.e., connected) parts corresponding to $I_1,\hdots,I_m$, ordered so that $i \in I_1$. 
	Then 
	$$
	v_{I,J}(\varpi_i) = v_{I_1, I_1^c}(\varpi_i).
	$$
\end{lemma}

\begin{proof}
Let $I' = \bigsqcup_{k=2}^m I_k$, so $I = I_1\sqcup I'$ and $W_I$ is the direct product of $W_{I_1}$ and $W_{I'}$. Therefore 
\begin{align*}
\frac{1}{|W_I|}\sum_{w\in W_I} w\varpi_i &= \frac{1}{|W_{I_1}||W_{I'}|}\sum_{u\in W_{I_1}} \sum_{v\in W_{I'}}uv\varpi_i\\
&= \frac{1}{|W_{I_1}||W_{I'}|}\sum_{u\in W_{I_1}} \sum_{v\in W_{I'}}u\varpi_i\\
&= \frac{1}{|W_{I_1}|}\sum_{u\in W_{I_1}}u\varpi_i.
\end{align*}
\end{proof}

Thus when parametrizing the vertices $v_{I,I^c}(\varpi_i)$, we need only consider the following cases: in the first case, $i \notin I$, which from Lemma \ref{mt} is just as good as $I = \emptyset$; in the second case, $i\in I$, and Lemma \ref{simpleton} lets us reduce this to $i\in I$ such that the subdiagram for $I$ is connected.

\begin{lemma}
Suppose $I\subseteq \{1,\hdots,r\}$ and that the subdiagram for $I$ is connected. Express 
$$
\varpi_i-v_{I,I^c}(\varpi_i) = \sum_{m\in I} c_m\alpha_m.
$$
Then each $c_m>0$. 
\end{lemma}

\begin{proof}
Note that $\varpi_i-v_{I,I^c}(\varpi_i)$ is the average of the terms $\varpi_i-w\varpi_i$, for $w\in W_I$. Each such term is a nonnegative combination of simple roots. So it suffices to find a single $w$ such that $\varpi_i-w\varpi_i$ is a positive combination of the simple roots $\alpha_m$ for $m\in I$. 

Let $\theta_I$ denote the highest root of the sub-root system given by $I$. Since $\varpi_i(\theta_I^\vee)>0$, $\varpi_i-s_{\theta_I}\varpi_i = \varpi_i(\theta_I^\vee)\theta_I$ is positively supported on every simple root $\alpha_m$, for $m\in I$. 
\end{proof}

\begin{corollary}\label{distinct}
	Let $I_1$ and $I_2$ be distinct subsets of $\{1,\hdots,r\}$, each containing $i$ and each with connected Dynkin subdiagram. Then $v_{I_1,I_1^c}(\varpi_i)\neq v_{I_2,I_2^c}(\varpi_i)$. 
\end{corollary}

\begin{proof}
Expressed in the basis of simple roots, $\varpi_i - v_{I_1,I_1^c}(\varpi_i)$ and $\varpi_i - v_{I_2,I_2^c}(\varpi_i)$ have different coefficients. 
\end{proof}	

Therefore we have just proved the following parametrization of the extremal rays of $\mathscr{K}(G)$. 

\begin{theorem}\label{BigTheorem}
The extremal rays of $\mathscr{K}(G)$ with first coordinate $\varpi_i$ coincide precisely with the set $\{(\varpi_i, v_{I,J}(\varpi_i))\}$, where either 
\begin{enumerate}[label=(\alph*)]  
 \item the Dynkin subdiagram for $I$ is connected and $i\in I$ (note $I = \{1,\hdots,r\}$ is one such choice),
 \item $I=\emptyset$, which corresponds to the extremal ray $(\varpi_i, \varpi_i).$ \\ \end{enumerate}
\end{theorem}

Proposition \ref{firstcoord} and Theorem \ref{BigTheorem} taken together give a complete description of the extremal rays of $\mathscr{K}(G)$.

\section{Lifting Extremal Rays from Levi Subgroups}\label{levilifting}

In this section, we present a new framework in which to view the vertices $v_{I,J}(\lambda)$ or the extremal rays $(\varpi_i, v_{I,J}(\varpi_i))$. Specifically, for each $I$, we describe maps of polytopes and of Kostka cones for the two groups $L\subset G$, where $L$ is the (semisimple) Levi subgroup associated with $I$. It would be interesting to understand the extent to which this ``Levi induction'' is related to the Levi induction developed in \cite{Belkale,BKiers}. 

Let $L$ be a choice of semisimple Levi subgroup 
of $G$, and let $I=\Delta(L)$ be the corresponding collection of simple root indices. For such a simple Levi subgroup we can identify the simple roots in $\Phi(L)$ with a subset of the simple roots in $\Phi(G)$. We write $\overline{x}_i$ for the fundamental coweights of $L$. Thus for $i,j\in \Delta(L)$, $\langle \alpha_j,x_i\rangle_G = \langle \alpha_j,\overline{x}_i\rangle_L$. 

\begin{definition}
	Given a Levi subgroup $L \hookrightarrow G$ and a dominant integral weight $\lambda_L \in \Lambda_L^+$, we associate to $\lambda_L$ a dominant integral weight $\lambda \in \Lambda_G^+$, as follows: if $i \in \Delta(L)$ then $\langle \lambda, \alpha^{\vee}_i \rangle = \langle \lambda_L, \alpha_i^{\vee} \rangle$ and if $j \notin \Delta(L)$ then $\langle \lambda, \alpha_j^{\vee} \rangle =0$. We call this new weight $\lambda$ the \emph{extension by 0} of $\lambda_L$.
\end{definition}

 Let $(\lambda_L, \mu_L)$ be an element of $\mathscr{K}(L)$, and write
\[
\lambda_L-\mu_L = \sum_{k \in \Delta(L)} c_k\alpha_k .
\]

 Then define the weight $\lambda$ for $G$, extending $\lambda_L$ by zero on each $\alpha_k^\vee, k\not\in \Delta(L)$, and set 
 \begin{equation} 
 \mu:=\lambda-\sum_{k \in \Delta(L)} c_k \alpha_k ,
 \end{equation}
 as a weight for $G$. 
 
 \begin{definition}
 	Given a point $(\lambda_L, \mu_L) \in (\Lambda_L^+)^2$ for $\lambda_L=\sum_L a_i \varpi_i$ and $\mu_L=\lambda_L-\sum_L c_i \alpha_i|_L$, we write $\mathrm{Ind}_L^G(\lambda_L, \mu_L)=(\lambda, \mu)\in (\Lambda_G^+)^2$ where $\lambda = \sum a_i\varpi_i$ is the extension of $\lambda_L$ by 0 and $\mu=\lambda-\sum c_i \alpha_i$.
 \end{definition}
 
We now record, without proof, several straightforward properties of $\op{Ind}_L^G$. 
 
 \begin{proposition}\label{liftingprop}
 $ $ 

 \begin{enumerate}[label=(\alph*)]
 \item The map $\op{Ind}_L^G: (\Lambda_L^+)^2 \rightarrow (\Lambda^+_G)^2$ restricts to $\op{Ind}_L^G: \mathscr{K}(L)\to \mathscr{K}(G)$.
 \item For any two Levi subgroups $L'\subset L\subset G$, we have $\op{Ind}_L^G \circ \op{Ind}_{L'}^L = \op{Ind}_{L'}^G$. 
 \item Given a vertex $\mu_L = v_{I',J'}(\lambda_L)$ of $IP_{\lambda_L}$ for a Levi $L$, we have 
 $$
 \op{Ind}_L^G ( \lambda_L, v_{I',J'}(\lambda_L)) = (\lambda, v_{I,J}(\lambda)),
 $$
 where $\lambda$ is the extension of $\lambda_L$ by $0$ and $I = I'$, $J = I^c \subseteq\{1,\hdots,r\}$. 
 \end{enumerate}
 \end{proposition}
 
Thus we observe that, if $\lambda$ is the extension by 0 associated to $\lambda_L$ for some simple Levi $L \hookrightarrow G$, then whenever $I$ is contained in the Dynkin subdiagram defining $L$, the vertex $v_{I,J}(\lambda)$ of $IP_{\lambda}$ is in fact induced from a vertex of $IP_{\lambda_L}$ for $L$. 

Furthermore, whenever we lift a vertex $v_{I_L, J_L}(\lambda_L)$ from a Levi subgroup it always lifts to $v_{I_L, I_L^c}(\lambda)$, so by the compatibility in Proposition \ref{liftingprop}(b), we can make the following notational simplification.
\begin{notation}
	We write $v_{\Delta(L)}(\lambda)$ for a vertex of $IP_\lambda$ lifted from Levi $L$, where $L$ is the smallest possible such Levi. 
\end{notation}

Before examining the situation with non-simple Levi subgroups, we review the notion of a direct sum of cones.

\begin{definition}
If $C_1, C_2$ are rational semigroups inside ambient vector spaces $W_1,W_2$, we may form the direct sum $C_1\oplus C_2 = \{(c_1,c_2)|c_1\in C_1, c_2\in C_2\}\subseteq W_1\oplus W_2$, which is again a semigroup under $(c_1,c_2)+(d_1,d_2)=(c_1+d_1,c_2+d_2)$ and admits scalar multiplication by nonnegative rational numbers $q\cdot (c_1,c_2) = (qc_1,qc_2)$. Furthermore, if $C_1$ and $C_2$ are convex rational cones (i.e., defined by a finite system of rational linear inequalities), so is $C_1\oplus C_2$ (defined by the union of the inequalities for $C_1$ and $C_2$). 
\end{definition}

\begin{lemma}
	Let $L$ be a non-simple Lie group, $\displaystyle L= \prod L_i$ where each $L_i$ is simple. Then $\mathscr{K}(L)=\bigoplus \mathscr{K}(L_i)$.
\end{lemma}

\begin{proof}
	If $L$ is non-simple then the root system $\Phi(L)= \bigsqcup \Phi(L_i)$. For any finite-dimensional representation $V_{\lambda_L}$ we have $V_{\lambda_L}=\bigboxtimes V_{\lambda_{L_i}}$; here each simple factor $L_i$ acts only on $V_{\lambda_{L_i}}$. The statement on Kostka cones follows. 
\end{proof}

\begin{remark}
In describing $\mathscr{K}(G)$, this allows us to consider only simple $G$, as was assumed in the introduction. 
\end{remark}

\begin{proposition}\label{levisum}
	Let $\displaystyle L= \prod L_k$ be a direct product of simple Levi subgroups. For each $k$, suppose we are given a dominant weight $\lambda_{L_k}$ and a vertex $\mu_{L_k} = v_{I_k,I_k^c}(\lambda_{L_k})$. Set $I = \bigcup I_k$. 
	Then \[\sum \mathrm{Ind}_{L_i}^{G}(\lambda_{L_i}, \mu_{L_i})=\left(\sum \lambda_i, v_{I,I^c}\left(\sum \lambda_i\right)\right)\] where each $\lambda_i$ is the extension by $0$ of $\lambda_{L_i}$.
\end{proposition}

\begin{proof}
	This follows by an argument very similar to the proof of Lemma \ref{simpleton}.
\end{proof}

With the notation of Levi induction, the extremal rays of $\mathscr{K}(G)$ take on a particularly nice form. 

\begin{corollary}\label{lift0}
 Let $L$ be a simple Levi with $i \in \Delta(L)$. The extremal ray $(\varpi_i, v_{\Delta(L)}(\varpi_i))=\mathrm{Ind}_{L}^{G}(\varpi_i^L, 0)$.
\end{corollary}

\begin{proof}
Simply note that $\langle 0, \alpha^{\vee}_k \rangle = 0$ for all $k \in I$, so $ 0 = v_{\Delta(L)}(\varpi_i^L)$. 

Alternatively, if $\lambda$ is the extension of $\lambda_L$ by zero, $(\lambda,v_{\Delta(L)}(\lambda))=\op{Ind}_{L}^G(\lambda_L,0)$ is clear from Proposition \ref{avg} and the observation that, in any root system $\Phi_L$ with any weight $\lambda_L$, 
$$
\sum_{w\in W_L} w\lambda_L = 0.
$$
\end{proof}

Suppose one would like to know the difference $\varpi_i - v_{\Delta(L)}(\varpi_i)$ in terms of simple roots (for example, this would help one to write down the extremal rays more quickly than using the averaging of Proposition \ref{avg}). This task is the same as expressing $\varpi_i^L - 0 $ in terms of simple roots for the Levi root system, by the previous corollary. But this is encoded by the transpose inverse Cartan matrix associated to $L$: the desired coefficients of the simple roots are the entries of the $i^\text{th}$ column of said matrix. 

\section{Detailed example in $C_4$}


In this section, we compute the extremal rays of the form $(\varpi_3, -)$ in type $G=C_4$. We will write down  lattice representatives of the extremal rays; they will be of the form $(k\varpi_3, v_{\Delta(L)}(k\varpi_3))$ for some integer $k$ (and where $3\in \Delta(L)$). Taking $k=det (C_L)$, where $C_L$ is the Cartan matrix for $L$, will always yield $k \varpi_i$ on the root lattice. Note however that $(det (C_L) \varpi_i^L, 0)$ may not be the \emph{first} integral point on the ray $(\varpi_i^L,0)$, as seen in case 4 below corresponding to Levi $\Delta(L)=\{2,3,4\}$.

In the examples below, $\Delta(L)$ are denoted by solid nodes, and a circle is placed around node $3$. The transpose inverse Cartan matrix of $L$ is given below the diagrams, with columns aligned with the corresponding simple roots of $L$. 

~\\

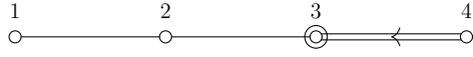
\begin{center}
\begin{dynkinDiagram}[root radius=.08cm, edge length=2cm, labels*={1,2,3,4}] C{oooo} \draw[black] (root 3) circle(.15cm); \end{dynkinDiagram} \vspace{1em}

The trivial Levi $\{e\}$ \vspace{1em}

$(\varpi_3, \varpi_3)$ \vspace{2em}

\end{center}

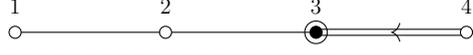
\begin{center}
\begin{dynkinDiagram}[root radius=.08cm, edge length=2cm, labels*={1,2,3,4}] C{oo*o} \draw[black] (root 3) circle(.15cm); \end{dynkinDiagram} \vspace{1em}

\hspace{1.6cm} $\frac{1}{2} \begin{pmatrix} 1 \end{pmatrix}$ \vspace{1em}

$(2 \varpi_3, v_{3, 124}(2\varpi_3))=(2\varpi_3, 2\varpi_2-\alpha_3)=(2\varpi_3, \varpi_2+\varpi_4)$ \vspace{2em}

\end{center}

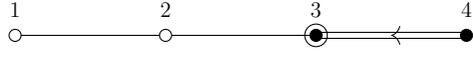
\begin{center}
\begin{dynkinDiagram}[root radius=.08cm, edge length=2cm, labels*={1,2,3,4}] C{oo**} \draw[black] (root 3) circle(.15cm); \end{dynkinDiagram} \vspace{1em}

\hspace{3.6cm} $\frac{1}{2} \begin{pmatrix} 2 & \hspace{1.5cm} 2 \\ 1 & \hspace{1.5cm} 2 \end{pmatrix}$ \vspace{1em}

$(2 \varpi_3, v_{34, 12}(2\varpi_3))=(2\varpi_3, 2\varpi_3-2\alpha_3-\alpha_4)=(2\varpi_3,  2\varpi_2)$ \vspace{2em}
\end{center}

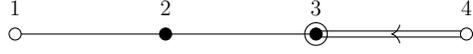
\begin{center}
\begin{dynkinDiagram}[root radius=.08cm, edge length=2cm, labels*={1,2,3,4}] C{o**o} \draw[black] (root 3) circle(.15cm); \end{dynkinDiagram} \vspace{1em}

\hspace{-.3cm}$\frac{1}{3} \begin{pmatrix} 2 & \hspace{1.5cm} 1 \\ 1 & \hspace{1.5cm} 2 \end{pmatrix}$ \vspace{1em}

$(3 \varpi_3, v_{23, 14}(3\varpi_3))=(3 \varpi_3, 3\varpi_3-\alpha_2-2\alpha_3)=(3 \varpi_3, \varpi_1+2 \varpi_4)$ \vspace{2em}
\end{center}

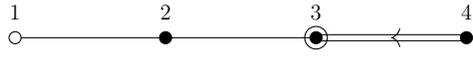
\begin{center}
\begin{dynkinDiagram}[root radius=.08cm, edge length=2cm, labels*={1,2,3,4}] C{o***} \draw[black] (root 3) circle(.15cm); \end{dynkinDiagram} \vspace{1em}

\hspace{1.7cm}$\frac{1}{2} \begin{pmatrix} 2 & \hspace{1.5cm} 2 & \hspace{1.5cm} 2  \\ 2 & \hspace{1.5cm} 4 & \hspace{1.5cm} 4 \\ 1 & \hspace{1.5cm} 2 & \hspace{1.5cm} 3 \end{pmatrix}$ \vspace{1em}

$(2 \varpi_3, v_{234, 1}(2\varpi_3))=(2\varpi_3, 2\varpi_3-2\alpha_2-4 \alpha_4-2\alpha_2)=(2 \varpi_3, 2 \varpi_1)$ \vspace{2em}
\end{center}

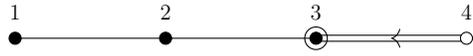
\begin{center}
\begin{dynkinDiagram}[root radius=.08cm, edge length=2cm, labels*={1,2,3,4}] C{***o} \draw[black] (root 3) circle(.15cm); \end{dynkinDiagram} \vspace{1em}

\hspace{-2.3cm}$\frac{1}{4} \begin{pmatrix} 3 & \hspace{1.5cm} 2 & \hspace{1.5cm} 1  \\ 2 & \hspace{1.5cm} 4 & \hspace{1.5cm} 2 \\ 1 & \hspace{1.5cm} 2 & \hspace{1.5cm} 3 \end{pmatrix}$ \vspace{1em}

$(4 \varpi_3, v_{123, 4}(4\varpi_3))=(4\varpi_3, 4\varpi_3-\alpha_1-2\alpha_2-3\alpha_3)=(4 \varpi_3, 3 \varpi_4)$ \vspace{2em}
\end{center}

\begin{center}
\begin{dynkinDiagram}[root radius=.08cm, edge length=2cm, labels*={1,2,3,4}] C{****} \draw[black] (root 3) circle(.15cm); \end{dynkinDiagram} \vspace{1em}

\hspace{-.3cm}$\frac{1}{2} \begin{pmatrix} 2 & \hspace{1.5cm} 2 & \hspace{1.5cm} 2 & \hspace{1.5cm} 2  \\ 2 & \hspace{1.5cm} 4 & \hspace{1.5cm} 4 & \hspace{1.5cm} 4 \\ 2 & \hspace{1.5cm} 4 & \hspace{1.5cm} 6 & \hspace{1.5cm} 6 \\ 1 & \hspace{1.5cm} 2 & \hspace{1.5cm} 3 & \hspace{1.5cm} 4 \end{pmatrix}$ \vspace{1em}

$(2 \varpi_3, v_{1234, \varnothing}(2\varpi_3))=(2 \varpi_3, 2\varpi_3-2 \alpha_1-4\alpha_2-6\alpha_3-3\alpha_4)=(2 \varpi_3, 0)$
\end{center}

\section{Polynomial Formulas for the Number of Extremal Rays}
We now give formulas for the number of extremal rays in $\mathscr{K}(G)$ for each type of simple Lie group $G$ as polynomials in the rank. In the classical setting of Kostka numbers (i.e., for $G=A_r$), these are due to Gao, Orelowitz, and Yong, along with the third author \cite{GKOY}. Our computations are based on the correspondence 
\begin{equation}\label{corresp}
\#\{\textit{extremal rays of the form } (k\varpi_i, v_{I, J}(k\varpi_i)) \} = 1+\#\{\textit{simple Levis } L \textit{ containing node } i\}
\end{equation}
as given by Theorem \ref{BigTheorem}. 

\begin{remark}\label{ADE} As we are looking at choices of simple Levi subgroups, the only influential factor is the underlying graph of the Dynkin diagram, and \textbf{not} the lacing of the diagram. Thus, we have that $G=A_r$, $B_r$, and $C_r$ will have the same number of extremal rays for each $r$, as will the pairs $A_2$ and $G_2$, and $A_4$ and $F_4$. We therefore reduce to computations for $A_r$, $D_r$, and $E_r$. We find it interesting that each polynomial begins with leading term $r^3/6$. 
\end{remark} 

We take as convention the numbering scheme of Bourbaki \cite{Bour} for the simple Dynkin diagrams. 
\begin{proposition} \begin{enumerate}[label=(\alph*)]
\item For $G$ of type $A_r$, the number of extremal rays in $\mathscr{K}(G)$ is $\binom{r+1}{3}+\binom{r+1}{2}+\binom{r+1}{1}-1$. 

\item For $G$ of type $D_r$, the number of extremal rays in $\mathscr{K}(G)$ is $\binom{r}{3}+3\binom{r}{2}+2\binom{r}{1}-3$.

\item For $G$ of type $E_r$, the number of extremal rays in $\mathscr{K}(G)$ is $\binom{r}{3}+4\binom{r}{2}+\binom{r}{1}-8$.

\end{enumerate}

\end{proposition}

\begin{proof}
\begin{enumerate}[label=(\alph*)]

\item We proceed by induction on $r$. Set $R_{i,r}$ to be the number of extremal rays of the form $(k\varpi_i, v_{I,J})$ for $A_r$. By the correspondence (\ref{corresp}), we have as the base case that $R_{1,1}=2$, which agrees with the formula (with the convention that $\binom{2}{3}=0$). Now suppose that the formula holds for $r$. Again using the correspondence (\ref{corresp}), we know that $R_{i,r+1}-R_{i,r}$ ($1 \leq i \leq r$) is precisely the number of new Levi subgroups in $A_{r+1}$ containing node $i$, under the usual embedding $A_r \hookrightarrow A_{r+1}$, and that $R_{r+1, r+1}=r+2$. As can easily be seen from the Dynkin diagram, we have for $1 \leq i \leq r$
$$
R_{i, r+1} -R_{i,r} = i.
$$
Therefore, we have 
$$
\begin{aligned} \sum_{i=1}^{r+1} R_{i,r+1} &= \left(\sum_{i=1}^{r} R_{i,r+1} \right) +R_{r+1, r+1} \\
&= \sum_{i=1}^r \left(R_{i,r}+i\right)+r+2 \\
&= \binom{r+1}{3}+\binom{r+1}{2}+\binom{r+1}{1}-1+\frac{r(r+1)}{2} + r+2\\
&= \binom{r+2}{3}+\binom{r+2}{2}+\binom{r+2}{1}-1
\end{aligned}
$$
as desired.

\item We again proceed by induction on $r$. Denote similarly in this case $R_{i,r}$. Then we have as a base case 
$$
R_{1,4}=R_{3,4}=R_{4,4}=6, \ R_{2,4}=9,
$$
which can be obtained using the correspondence (\ref{corresp}), and agrees with the formula. Now, choose the standard Levi embedding $D_r \hookrightarrow D_{r+1}$, noting that now node $i$ is labeled $i+1$. An investigation of the Dynkin diagram in this case gives $R_{1, r+1} = r+3$ and  
$$
R_{i,r+1} = \begin{cases} 
	R_{i-1,r}+((r+1-i)+2), & 2\leq i \leq r-1 \\ 
	R_{r-1,r}+2, & i=r \\
	 R_{r,r}+2, & i=r+1
	 \end{cases}
$$
Therefore, we have 
$$
\begin{aligned}
\sum_{i=1}^{r+1} R_{i,r+1} &= r+3 + \left(\sum_{i=2}^{r-1} R_{i,r+1}\right) +R_{r,r+1}+R_{r+1,r+1}\\
&= r+3+ \left(\sum_{i=2}^{r-1} R_{i-1,r} +r+3-i \right) +R_{r-1,r}+2+R_{r,r}+2 \\
&=r+7 +\sum_{i=1}^r R_{i,r} + \sum_{i=2}^{r-1} (r+3-i) \\
&= r+7 +\binom{r}{3}+3\binom{r}{2}+2\binom{r}{1}-3+\frac{1}{2}(r^2+3r-10)\\
&=\frac{r^3}{6}+\frac{3r^2}{2}+\frac{10r}{3}-1 \\
&= \binom{r+1}{3}+3\binom{r+1}{2}+2\binom{r+1}{1}-3
\end{aligned}
$$
as desired. Note that this formula is still valid for $D_3 = A_3$ and $D_2 = A_1\times A_1$.

\item This is done by direct computation. In particular, the number of extremal rays in types $E_6$, $E_7$, and $E_8$ are $78$, $118$, and $168$, respectively, which fits the formula. It should be noted that this also holds when making the associations $E_5=D_5$, $E_4=A_4$, and $E_3 = A_2\times A_1$.
\end{enumerate}
\end{proof}

%
%

\end{document}